\documentclass[10pt]{article}
\usepackage{amsmath}
\usepackage{amssymb}
\usepackage{amsthm}
\usepackage{verbatim}
\usepackage{tikz}
\title{Local semicircle law at the spectral edge for Gaussian $\beta$-ensembles}
\author{Percy Wong}

\begin{document}
\maketitle
\newtheorem{thm}{Theorem}
\newtheorem{prop}{Proposition}
\newtheorem{lem}{Lemma}
\newtheorem{cor}{Corollary}
\newtheorem{defn}{Definition}
\newtheorem{rmk}{Remark}
\begin{abstract}
We study the local semicircle law for Gaussian $\beta$-ensembles at the edge of the spectrum.  We prove that at the almost optimal level of $n^{-2/3+\epsilon}$, the local semicircle law holds for all $\beta \geq 1$ at the edge.  The proof of the main theorem relies on the calculation of the moments of the tridiagonal model of Gaussian $\beta$-ensembles up to the $p_n$-moment where $p_n = O(n^{2/3-\epsilon})$.  The result is the analogous to the result of Sinai and Soshnikov \cite{SS} for Wigner matrices, but the combinatorics involved in the calculations are different.  
\end{abstract} 

\section{Introduction and Summary}

Our goal in this paper is to prove a local semicircle law at the level of
$O(n^{2/3-\epsilon})$ near the spectral edges of the Gaussian
$\beta$-ensembles.  The corresponding problem for Wigner matrices at the edge has been established by Sinai and Soshnikov \cite{SS} and in the bulk by Erd\"os, Schlein and Yau \cite{ESY}, \cite{ESY2}.

\begin{defn}
Let $\beta \geq 1$, the Gaussian $\beta$-ensemble is an ensemble in $\mathbb{R}$ that have the following probability density:
\begin{equation} \label{eqn:beta}
 df(\lambda_1,\hdots,\lambda_n) = G_{n,\beta} \prod_{1\leq i<j\leq n} |\lambda_i
- \lambda_j|^\beta
\exp(-\frac{\beta}{4}\sum_1^n \lambda_i^2) \prod_i d\lambda_i 
\end{equation}
where $G_{n,\beta}$ is a normalization constant.
\end{defn}

For $\beta = 1,2,4$, this probability distribution corresponds to the
eigenvalues of GOE, GUE and GSE respectively.  Other known and related
results regarding the $\beta$-ensembles will be discussed in the next
section.  Using the moments method, one can establish the Wigner semicircle law
for the $\beta$-ensembles, e.g. \cite{D}:

Suppose $\lambda_i, i = 1,\hdots,n$ follow the distribution of a
$\beta$-ensemble.  Consider the rescaled eigenvalues $\tilde{\lambda_i} =
\frac{\lambda_i}{2\sqrt{n\beta}}$ and the empirical distribution function:
\begin{displaymath}
 N_n(\lambda) = \frac{1}{n}\#\{k:\tilde{\lambda_k} <
\lambda\}
\end{displaymath}
We have
\begin{displaymath}
 \lim_{n\rightarrow\infty} N_n(\lambda) = \int_{-\infty}^\lambda \rho(u)du
\end{displaymath}
where
\begin{displaymath}
 \rho(u) = \left \{ \begin{array}{ll}
                     0 & \textrm{for } u > 1 \textrm{ or } u < -1, \\
		    \frac{2}{\pi}\sqrt{1-u^2} & \textrm{for } -1 \leq u \leq 1.
                    \end{array} \right.
\end{displaymath}

Take an $r_n$-neighborhood $O_n$ of right end of the spectrum $\lambda =
1$, where $r_n = O(n^{-2/3 + \epsilon})$  for some $\epsilon > 0$. 
Normalizing $\tilde{\lambda}_k = 1 - \theta_k r_n$ and 
\begin{displaymath}
 \mu_n(\theta_k) = \frac{1}{nr_n^{3/2}} 
\end{displaymath}
at each $\theta_k$, we obtain a measure $\mu_n$ on the real line such that
$\mu_n(\mathbb{R}) = r_n^{-3/2}$.  The main result of this paper is the proof of the following local semicircle law at the edge of the spectrum:

\begin{thm} \label{thm:main}
As $n\rightarrow \infty$, $\mu_n$ converge weakly in probability on each
finite interval to a measure $\mu$ concentrated on $\mathbb{R}_+$ and almost
continuous with respect to the Lesbesgue measure:
\begin{displaymath}
 \frac{d\mu}{dx} = \left \{ \begin{array}{ll} 
                            \frac{2\sqrt{2}}{\pi}\sqrt{x} & \textrm{if $x>0$}\\
			    0 & \textrm{otherwise}
                            \end{array} \right.
\end{displaymath}
If $r_n > n^{\epsilon-2/3}$ for some $\epsilon > 0$, then the measures converge
vaguely to $\mu$ with probability $1$.
\end{thm}

\begin{rmk}
The reason we called the above theorem a local semicircle law is that under the semicircle law, one can show that for a neighborhood of $1$ of size of some fixed number $\delta$, the proportion of normalized eigenvalues in that neighborhood converges to the corresponding integral of the semicircle law.  The main theorem states that the size of this neighborhood can be much finer, at the scale of $n^{-2/3+\epsilon}$ and the convergence will still hold.  In other words, for a neigborhood of size $r_n$ near $1$, there are $\Omega{nr_n^{3/2}}$ eigenvalues.  If one rescales the total number in this interval by $nr_n^{3/2}$, this number with converge in probability to the appropriate integral of the semicircle law.  Note that one would not expect the semicircle shape to hold at the edge when $r_n = n^{-2/3}$.  In this sense, the local semicircle law above is almost optimal.
\end{rmk}

The theorem above will be proven using the moment method, similar to the method by Sinai and Soshnikov \cite{SS}:
\begin{thm} \label{thm:mom}
 Let $A_n$ be a random matrix whose eigenvalues are given by
equation (\ref{eqn:beta}) scaled by a factor of $1/2\sqrt{n\beta}$. 
Let $p_n \rightarrow \infty$ as $n\rightarrow\infty$ and $p_n =
O(n^{2/3-\epsilon})$ for some $\epsilon > 0$,
then
\begin{displaymath}
 \mathbb{E}(\textrm{Tr} A_n^{p_n}) = \left \{ \begin{array}{ll}
					     2^{3/2}n{(\pi p_n^3)}^{-1/2}
(1+o(1)) & \textrm{if $p_n$ is even,} \\
 0, & \textrm{if $p_n$ is odd.}
                                            \end{array} \right.
\end{displaymath}

\end{thm}

However we shall see in the proof that the combinatorics involved
is different and simpler in our case.  The derivation of theorem
\ref{thm:main} from theorem \ref{thm:mom} is the same as in \cite{SS} and will
be given in the appendix.  It should also be remarked that the result in this paper is slightly weaker than that in \cite{SS}, as in their paper, only $r_n << n^{2/3}$ is required.

The structure of the paper is the following.  We first introduce a
matrix model, proposed by Dumitriu and Edelman \cite{DE}, for the
$\beta$-ensembles and recall some related results.  We shall then prove
theorem \ref{thm:mom} in section \ref{sec:proof}.

For the rest of the paper, we shall use the following notion of an event depending on some index $n$ having overwhelming probability:
\begin{defn}
We say an event $E$ holds with overwhelming probability if for all $n$, $\mathbb{P}(E) \geq 1 - O_C(n^{-C})$ for every constant $C$.
\end{defn}
It should be observed that a union of $n^k$ events of overwhelming probability for some fixed $k$ still holds with overwhelming probability.

\section{Gaussian Beta Ensembles}
As mentioned above, the special values of $\beta = 1,2,4$ corresponds to
eigenvalues of GOE, GUE and GSE respectively and there is an extensive
literature on them.  For the edge of the spectrum, Tracy and
Widom \cite{TW}, \cite{TW2} has been able to establish that, upon rescaling and
centering,
the top eigenvalue distribution converges to what is known as the Tracy-Widom
distribution.  Later Soshnikov \cite{S} proved the edge universality for Wigner
matrices.  However, the case of general $\beta$ remains unknown until the
recent paper by Ramirez, Rider and Virag \cite{RRV}.  One of the obstacles for
studying $\beta$-ensembles previously is the lack of matrix formulation for
them.  In their paper \cite{DE}, Dumitriu and Edelman succeeded in writing down
a tridiagonal random matrix model whose eigenvalue distribution is given by equation
(\ref{eqn:beta}):

\begin{thm} [Dumitriu-Edelman]
 Consider the matrix given by 
\begin{equation} \label{eqn:mat}
 A_{n,\beta} = \left( \begin{array}{ccccccc}
               N(0,2) & \chi_{(n-1)\beta} & 0 & \ldots & \ldots & \ldots &  0 \\
	       \chi_{(n-1)\beta} & N(0,2) & \chi_{(n-2)\beta} & 0 & \ldots &
\ldots & 0 \\
		0 & \chi_{(n-2)\beta} & N(0,2) & \chi_{(n-3)\beta}
& 0 & \ldots & 0 \\
		\vdots & \ddots & \ddots & \ddots & \ddots & \ddots & \vdots \\
	      0 & \ldots & 0 & \chi_{3\beta} & N(0,2) & \chi_{2\beta} & 0 \\
	      0 & \ldots & \ldots & 0 & \chi_{2\beta} & N(0,2) & \chi_\beta \\
	      0 & \ldots & \ldots & \ldots & 0 & \chi_\beta & N(0,2)
              \end{array} \right)
\end{equation}
where $N(0,2)$ denotes a random variable whose distribution follows the Gaussian
distribution with mean $0$ and variance $2$; $\chi_k$ is a random variable having a chi distribution with $k$ degree of freedom.  The upper triangular
part of the matrix consists of independent random variables and the matrix is
symmetric.  The eigenvalues of this matrix follows the distribution of the
$\beta$-ensemble.
\end{thm}

For GOE, the theorem can be easily proven by successive Householder
transformation and this was observed by Trotter \cite{T}. For general $\beta$,
the idea of the proof is to compute the Jacobian of the diagonization of the
matrix of the form (\ref{eqn:mat}).

Another important result is the one by Ramirez, Rider and Virag \cite{RRV}, whose work is based on previous work by Edelman and Sutton \cite{ES}.  In their paper,
Ramirez, Rider and Virag proved that the edge scaling parameter is $n^{2/3}$ and
computed the asymptotics of the rescaled limiting distribution of the
$\beta$-ensembles:

\begin{thm}[Ram\'irez, Rider, Virag]
 Let $\lambda_1,\hdots, \lambda_k$ be the $k$ largest eigenvalue of the $\beta$-ensemble, in descending order, and
\begin{displaymath}
\mathcal{H}_\beta = -\frac{d^2}{dx^2}+x+\frac{2}{\beta}b'_x
\end{displaymath}
where $b'$ denote a white noise.  Let $\Lambda_1, \hdots, \Lambda_k$ be the $k$ lowest eigenvalues of $\mathcal{H}_\beta$, then 
\begin{displaymath}
 n^{1/6} (\lambda_n - \sqrt{2n\beta})
\end{displaymath}
converges in distribution to $(\Lambda_1,\hdots,\Lambda_k)$
Moreover the asymptotics of $\Lambda_1$ are given by:
\begin{displaymath}
\mathbb{P}(\Lambda_1 > a) = \exp(-\frac{2}{3}\beta a^{3/2}(1+o(1))) 
\end{displaymath}
and
\begin{displaymath}
 \mathbb{P}(\Lambda_1 < -a) = \exp(-\frac{1}{24}\beta a^3 (1+o(1)))
\end{displaymath}
\end{thm}

The proof of the above theorem is based on the observation that the tridiagonal
matrix model can be viewed as a discretization of certain stochastic
differential operator $\mathcal{H}_\beta$ and the probability distribution
of the largest eigenvalue of the tridiagonal matrix will converge to the
probability distribution of the eigenvalues associated stochastic differential operator as $n\rightarrow\infty$.  This point process is given by the so-called $Airy_\beta$ process and upon rescaling will converge to $\mu$ in our main theorem.  Notice that if we consider the operator $\mathcal{H}_\infty$, which is the free Laplacian, then the spectral measure of the operator under proper rescaling is $\mu$.

Of course, these two results are only the tip of the iceberg of the vast field
of random matrices, there are excellent monographs , e.g. \cite{AGZ}, \cite{F}
that discuss the latest developments in details.

We see from the above theorem that, at the scaling of $n^{2/3}$, the limiting
distributions are different for each $\beta$ at the edge of the spectrum.  On
the other hand, the Wigner semicircle law is universal for all $\beta$, one can
therefore ask the question at which scale does the edge look different for
different $\beta$, which leads us to consider the problem at hand.  Moreover,
one can also ask, given the availability of a matrix representation, whether
the classical moment method will be able to give us information at the edge of
the spectrum.  It is also part of the goal of this paper to explore this
possibility.  Lastly, a local semicircle law is always of interest. The power ofa local semicircle law can be seen in recent papers of
Erd\"os, Schlein and Yau \cite{ESY}, \cite{ESY2} and in Tao and Vu
\cite{TV} in their proof of universality in Wigner ensembles and also in the recent paper of Bourgarde, Erd\"os and Yau \cite{BEY}.  Bao and Su \cite{BS} has also established through the Brownian Carousel a local semicircle law in the bulk for the Gaussian $\beta$-ensemble.  Also, one does not
expect to prove a local result at the bulk using the moments as higher moments
highlight the eigenvalues at the edge.

\section{Combinatorial considerations} \label{sec:proof}
Our goal is to understand the trace of higher degrees of the matrix
$A_{n,\beta}$.  For simplicity of notations, we consider the matrix
scaled by $1/\sqrt{\beta}$, so that the off diagonal entries have variance $n,
n-1, \hdots, 1$ and the diagonal entries have variance $2/\beta$.  The only
difference is that we have to normalize the resulting
calculations of moments by $2^{p_n}n^{p_n/2}$.  The traces are given by

\begin{equation} \label{eqn:sum}
 \mathbb{E}(\textrm{Tr} A_{n,\beta}^k) = \mathbb{E} \sum_{\mathcal{P}}
\xi_{i_0i_1}\xi_{i_1i2}\hdots\xi_{i_{k-1}i_0}
\end{equation}

where $\mathcal{P}$ is the set of indices $\{i_0,i_1,\hdots,i_{k-1}\}$

It is customary to encode this information as a graph and to study
the combinatorics associated with this graph: 

\begin{figure}[h!]
\centering
\begin{tikzpicture}
 \filldraw[black] (0,0) circle (2pt) node[anchor = north] {1}
		  (1,0) circle (2pt) node[anchor = north] {2}
		  (2,0) circle (2pt) node[anchor = north] {3}
		  (3,0) circle (0.5pt)
		  (3.5,0) circle (0.5pt)
		  (4,0) circle (0.5pt)
		  (5,0) circle (2pt) node[anchor = north] {n-2}
		  (6,0) circle (2pt) node[anchor = north] {n-1}
		  (7,0) circle (2pt) node[anchor = north] {n};
 \draw (0,0)--(1,0)--(2,0)--(2.5,0);
 \draw (4.5,0)--(5,0)--(6,0)--(7,0);
 \draw (0,0) .. controls (-1,1) and (1,1) .. (0,0);	
 \draw (1,0) .. controls (0,1) and (2,1) .. (1,0);
 \draw (2,0) .. controls (1,1) and (3,1) .. (2,0);
 \draw (5,0) .. controls (4,1) and (6,1) .. (5,0);
 \draw (6,0) .. controls (5,1) and (7,1) .. (6,0);
 \draw (7,0) .. controls (6,1) and (8,1) .. (7,0);
\end{tikzpicture}
\caption{Graph representation of a tridiagonal matrix}
\end{figure}
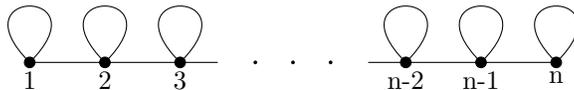

We first note an easy simplification of our problem at hand, due to Dumitriu
(see \cite{D} for the proof):
\begin{lem} \label{lem:simp}
 $\mathbb{E}(\emph{\textrm{Tr}}A_{n}^k) = n \mathbb{E}(A_n^k)_{11}$
\end{lem}

The plan now is to first calculate the contributions from paths that do not
go through self loops, i.e. paths where $i_j \neq i_{j+1}$ for all $j$; after
that we shall prove that the other paths contribute asymptotically lower order
terms to the sum.

Let us discuss of the odd moments first.  Even though the offdiagonal entries are not symmetric distributions and have nonzero odd moments.  One notices that the coefficients of the characteristics equation for the matrix involves only even powers of the offdiagonal entries, so we can replace the $\chi$-distribution by a symmetric distribution that has the same even moments as the $\chi$-distribution and the joint distribution of eigenvalues of the resulting matrix will be the same as that of the original.  Since the distribution of the trace depends only on that of the eigenvalues, we see that the odd moments vanish.  If $f(x)$ is the density of the $\chi$-distribution, then the distribution $g(x)$ equals to $f(x)/2$ for $x>0$ and $f(-x)/2$ for $x<0$ will be the desirable replacement.

We now return to the calculations of the even moments.  Associate with each path
on the graph a Dyck path of length $2k$, where
travelling to the right corresponds to an up edge and travelling to the left
corresponds to a down edge.  Throughout the rest of the paper, we shall
most of the time think of Dyck paths as probabilistics objects, i.e.
samples from a random walk $X$ of length $2k$ conditioned on $X_0 = X_{2k} = 0$
and $X_m \geq 0$ for $0\leq m \leq 2k$. 

Before we start the proof in earnest, let us first take a step back and recall some relevent known facts and see what could cause problems:

1) The number of Dyck paths of length $2k$ is given by the $k$-th Catalan number
$C_k$.
2) The $2l$-th moment of a chi-distributed random variable of degree with freedom
$k$ is given by $\prod_{i=0}^l (k + 2i)$.

Each path in the sum (\ref{eqn:sum}) will contribute $\prod_{i=0}^k (n + j_i)$ for
some $j_i$.  The hope is that asymptotically, for $k = o(n^{2/3})$, the product
will be $n^k (1 + o(1))$ and the worry is that some of the $j_i$ gets too small
(negative) or too big.  This corresponds to the case where the path goes too
far to the right or stay on the same edge on the graph for too many steps.  Therfore our goal is to show that there are very few of these problematic paths
that they do not affect the total sum as $n$ goes to infinity.

The first step is to connect Dyck paths with a true random walk, so that we can
use many of the known results and nice properties of the latter.  We start
with the following well-known lemma, which establishes the relationship between Dyck paths and Bernoulli bridges.

\begin{lem} \label{lem:BEBB}
 There is a one-to-one correspondance between $\{1,\hdots,k+1\} \times
\textrm{\{Dyck paths of length 2k\}} \leftrightarrow \textrm{\{Bernoulli
Bridges of length 2k\}}$
\end{lem}

\begin{proof}
 The proof is to provide an explicit map.  Attach a down path to the end of a
Dyck path, resulting in a total of $k+1$ down paths and $k$ up paths.  Pick one
of the $k+1$ down paths and make a cut before the down path.  The Bernoulli
bridge (from $(0,0)$ to $(2k,-1)$ with first path always downwards) is
contructed by moving the second half to the origin and attaching the first half
to it; 

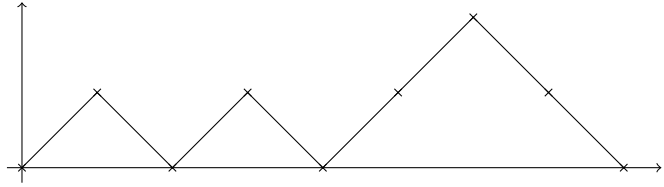
\begin{figure}[h]
\centering
\begin{tikzpicture} 
 \draw[->] (-0.2,0) -- (8.5,0); 
 \draw[->] (0,-0.2) -- (0,2.2);
 \draw plot[mark=x] coordinates {(0,0) (1,1) (2,0) (3,1) (4,0) (5,1) (6,2) (7,1)
(8,0)};
\end{tikzpicture}
\caption{A sample Dyck path}
\end{figure}
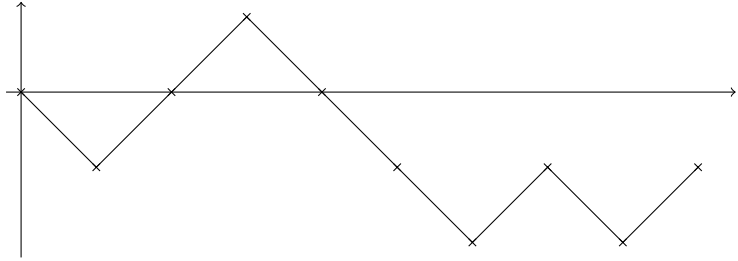
\begin{figure}[h]
 \centering
 \begin{tikzpicture}
  \draw[->](-0.2,0) -- (9.5,0);
  \draw[->](0,-2.2) -- (0,1.2);
  \draw plot[mark=x] coordinates {(0,0) (1,-1) (2,0) (3,1) (4,0) (5,-1) (6,-2)
(7,-1) (8,-2) (9,-1)};
 \end{tikzpicture}
 \caption{The Bernoulli bridge corresponding to the above Dyck path cutting at
the second down edge}
\end{figure} 
\end{proof}

The next lemma, relating the probability of an event of a Bernoulli bridge with
one of a random walk, is first proven by Khorunzhiy and Marckert \cite{KM}:

\begin{lem} [Khorunzhiy, Marckert] \label{lem:BBRW}
 Let $S_k = X_1 + \hdots + X_k$ be a simple random walk.  Let $\mathcal{F}_l$
be the measurable events generated by $X_1,\hdots, X_l$ and let
$\mathbb{P}_l^w$ be the probability measures induced on the paths by
$X_1,\hdots, X_l$.  Let $\mathbb{P}_l^b$ be the measure $\mathbb{P}_l^w$
conditioned on $X_0 = 0$ and $S_l = -1$.  Then for any $\mathcal{F}_k$-
measureable event
$A_k$,
\begin{equation}
 \mathbb{P}_{2k+1}^b(A_k) \leq C \mathbb{P}_{2k+1}^w(A_k) = C
\mathbb{P}_k^w(A_k)
\end{equation}
 for some constant $C$ independent of $n$.
\end{lem}

\begin{proof}
 The second equality is obvious as $A_k$ is $\mathcal{F}_k$-measurable. 
Suppose $A$ is $\mathcal{F}_k$-measurable, then 
\begin{displaymath}
 \mathbb{P}_{2k+1}^w (A | S_k = m, S_{2k+1} = -1) =
\mathbb{P}_{2k+1}^w(A|S_k=m) = \mathbb{P}_k^w(A|S_k=m),
\end{displaymath}
thus,
\begin{align} \nonumber
 \mathbb{P}_{2k+1}^b(A) &= \mathbb{P}_{2k+1}^w(A|S_{2k+1}=-1) \\ \nonumber
&= \sum_m \mathbb{P}_k^w(A|S_k=m) \mathbb{P}_{2k+1}^w(S_k=m|S_{2k+1}=-1).
\nonumber
\end{align}
On the other hand,
\begin{align} \nonumber
 \mathbb{P}_{2k+1}^w(S_k=m|S_{2k+1}=-1) &= \frac{\binom{k}{m}
\binom{k+1}{m+1}}{\binom{2k+1}{m+1}} \\ \nonumber
&= \frac{\binom{k}{m}}{2^k} \frac{2^k\binom{k+1}{m+1}}{\binom{2k+1}{m=1}} \\
\nonumber
&\leq \mathbb{P}_k^w(S_k=m) C_0, \nonumber
\end{align}
where $C_0 := \sup_k \sup_m \frac{2^k\binom{k+1}{m+1}}{\binom{2k+1}{m=1}}$
is finite.  It follows that 
\begin{align} \nonumber
 \mathbb{P}_{2k+1}^b(A) &\leq C_0 \sum_m
\mathbb{P}_k^w(A|S_k=m)\mathbb{P}_k(S_k=m) \\ \nonumber
&= C_0 \mathbb{P}_k^w(A) \nonumber
\end{align} 
\end{proof}

As discussed earlier, we would like to show that the probability of crossing
too far to the right of the graph is small and this is indeed the case:

\begin{prop} \label{prop:hor}
 Let $S_{2k}$ be a random walk conditioned on $S_0 = S_{2k} = 0$ and $S_l \geq
0$ for all $0 \leq l \leq 2k$.  Then for all $k$ 
\begin{equation}
 \mathbb{P}(\max_{0\leq l \leq 2k} S_k  \geq \lambda k^{1/2}) \leq C
\exp(-c\lambda^2)
\end{equation}
for some constants $c$ and $C$.
\end{prop}

\begin{proof}
Using the notations from the previous lemma, denote the Bernoulli
excursion of length $2k$, the Bernoulli bridge from $(0,0)$ to $(2k,0)$, and
the random walk of length $2k$ by $S_{2k}^e$, $S_{2k}^b$, $S_{2k}^w$
respectively.

By lemma \ref{lem:BEBB}, we have
\begin{displaymath}
 \max_{0\leq l\leq 2k}S_{l}^b - \min_{0\leq l\leq 2k}S_{l}^b - 1 \leq
\max_{0\leq l\leq 2k}s_l^e \leq \max_{0\leq l\leq 2k}S_{l}^b - \min_{0\leq l\leq
2k}S_{l}^b + 1
\end{displaymath}
So it suffices to show that
\begin{displaymath}
 \mathbb{P}(\max_{0\leq l\leq 2k}S_k^b - \min_{0\leq l\leq 2k}S_k^b \geq
\lambda k^{1/2}) \leq C \exp(-c\lambda^2)
\end{displaymath}
Moreover, if 
\begin{displaymath}
 \max_{0\leq l\leq 2k}S_k^b - \min_{0\leq l\leq 2k}S_k^b \geq
\lambda k^{1/2},
\end{displaymath}
then either
\begin{equation} \label{eqn:excep1}
 \max_{0\leq l\leq k}S_k^b - \min_{0\leq l\leq k}S_k^b \geq
\frac{1}{2}\lambda k^{1/2}
\end{equation}
or
\begin{displaymath}
 \max_{k+1\leq l\leq 2k}S_k^b - \min_{k+1\leq l\leq 2k}S_k^b \geq
\frac{1}{2}\lambda k^{1/2}.
\end{displaymath}
By reversing the Bernoulli bridge, it suffices to show the first inequality. 
The event satisfying the inequality (\ref{eqn:excep1}) is contained in the union
of the events
\begin{displaymath}
 \max_{0\leq l \leq 2k}S_k^b \geq \frac{\lambda}{4}k^{1/2} \textrm{ and }
\min_{0\leq l \leq 2k}S_k^b \leq \frac{-\lambda}{4}k^{1/2}
\end{displaymath}
By lemma \ref{lem:BBRW},
\begin{align} \nonumber
 \mathbb{P}(\max_{0\leq l \leq k} S_k^b \geq \frac{-\lambda}{4}k^{1/2}) &\leq
C_0 \mathbb{P}(\max_{0\leq l \leq k} S_k^w \geq \frac{-\lambda}{4}k^{1/2}) \\
& \leq 2C_0 \mathbb{P}(S_k \geq \frac{-\lambda}{4}k^{1/2})
\end{align}
where the second inequality is implied by the reflection principle. 
Hoeffding's inequality then implies the desired inequality:
\begin{displaymath}
 \mathbb{P}(\max_{0\leq l \leq k} S_k^b \geq \frac{-\lambda}{4}k^{1/2}) \leq
C\exp(-c\lambda^2)
\end{displaymath} 
\end{proof}

Next, we shall consider the potential problem of a path staying on the
same edge too many times.  Again, the probability of this event is
very small:

\begin{prop} \label{prop:ver}
 Let $S_l, 0\leq l \leq 2k$ be the Bernoulli excursion as above.  Let $T_i$
denote the number of $l$ such that $S_l = i$, for $i = 1,\hdots,2k$, then
\begin{displaymath}
 \mathbb{P}(\max_i T_i \geq k^{1/2+\epsilon}) \leq C\exp(-c
k^{2\epsilon^{\prime}})
\end{displaymath}
for any $\epsilon > \epsilon^{\prime} > 0$ and some constants $C, c$ independent
of $k,\lambda$.
\end{prop}

\begin{proof}
 As in the previous proposition, we shall connect the quantity in question with
a similar quantity of a random walk.  Clearly, since by mapping Dyck paths to
Bernoulli bridges, the times one visits a level is mapped to at most two
different levels, we have
\begin{displaymath}
 \mathbb{P}(\max_i T_i^e \geq \lambda k^{1/2}) \leq \mathbb{P}(2\max_i T_i^b
\geq \lambda k^{1/2})
\end{displaymath}
and as before we split the bridge into two halves
\begin{displaymath}
 \mathbb{P}(2\max_i T_i^b \geq \lambda k^{1/2}) \leq \mathbb{P}(\max_{0\leq
i\leq k} T_i^b \geq \frac{\lambda}{4}k^{1/2}) + \mathbb{P}(\max_{k+1\leq i\leq
2k} T_i^b \geq \frac{\lambda}{4}k^{1/2})
\end{displaymath}
By lemma \ref{lem:BBRW}, it suffices to show that
\begin{displaymath}
 \mathbb{P}(\max_{0\leq i \leq k} T_i^w \geq \frac{\lambda}{4}k^{1/2}) \leq
C\exp(-c\lambda^2)
\end{displaymath}
We will in fact first show that, for each $-k \leq 0 \leq k$,
\begin{displaymath}
 \mathbb{P}(T_i^w \geq \frac{\lambda}{4}k^{1/2}) \leq C\exp(-c\lambda^2)
\end{displaymath}
Let $q_{2k,r}$ be the probability that a random walk of length $2k$ returns to
$0$ exactly $r$ times.  It is well known (e.g. \cite{R}) that
\begin{displaymath}
 q_{2k,r} = \frac{1}{2^{2k-r}}\binom{2k-r}{k}
\end{displaymath}
We can then approximate the probability that there are less than
$k^{1/2+\epsilon}$ returns to $0$ for a random walk of length $2k$ using
Stirling's formula
\begin{displaymath}
 \mathbb{P}(\textrm{there are less than }k^{1/2+\epsilon}\textrm{ returns to
}0) \sim \sqrt{\frac{2}{\pi}}\int_0^{k^{1/2+\epsilon}} e^{-t^2/2}dt
\end{displaymath}
For levels other than $0$, we can condition on the first hitting time and
similarly show that
\begin{displaymath}
 \mathbb{P}(T_i \geq k^{1/2+\epsilon}) \leq C \exp(-ck^{2\epsilon})
\end{displaymath}
Therefore, we have the bound
\begin{displaymath}
 \mathbb{P}(\max_i T_i \geq k^{1/2+\epsilon}) \leq C
\exp(-ck^{2\epsilon^\prime})
\end{displaymath}
for any $\epsilon^{\prime} < \epsilon$ 
\end{proof}

With the two propositions above, we are ready to compute the contribution by
the main paths.  First of all the number of Dyck paths of length $2k$ is given
by the Catalan number $C_k = \frac{1}{k+1}\binom{2k}{k}$.  We have shown that
$\frac{1}{k+1}\binom{2k}{k}(1-o(1))$ of these paths have 'good properties',
i.e. they are not in the exceptional set in either propositions.  The
contribution of each paths is of order $n^k(1+o(1))$ when $k = o(n^{2/3})$.
To see this, we simply compute the contribution
\begin{displaymath}
 (n+i_1)\hdots(n+i_k)
\end{displaymath}
for that particular path.  Not being in the exceptional set in the first
proposition means that $i_j \geq -n^{1/3-\delta}$ for any $\delta > 0$ and not
being in the exceptional set in the second proposition means that $i_j \leq
n^{1/3-\delta}$ for any $\delta > 0$.  Therefore $\sum_j i_j \leq n^{1-\delta}$
and the claim follows.

It remains to show that the other paths do not contribute asymptotically to
the sum.  Let us first tackle the Dyck paths that do not have 'good
properties'.  Consider the paths where $\max_i T_i$ is of the order $q_n$ where
$q_n = p_n^{1/2 + l}$ for some $l>0$, each of the paths will contriubute
$O(n^{p_n}(1+\exp(n^{l})))$.  By the previous proposition, the
probability density of these paths are asymptotically of the order
$O(\exp(-n^{2l-\delta}))$.  Integrating over these paths, we can conclude
that they are asymptotically of $o(n^{p_n})$.  

Lastly we compute the paths where there are at least one loop, i.e., on the path
$\{1,i_1,i_2,\hdots,i_{p^n-1},1\}$, there exists $l$ such that $i_l = i_{l+1}$:

\begin{prop}
The contribution to the sum from paths with $2k$ loops, i.e. paths $\{1,i_1,i_2,\hdots,i_{p_n-1},1\}$ where there exists $l_1,\hdots,\l_{2k}$ such that $i_{l_j}=i_{l_j+1}$ for $j =1,\hdots,2k$ is of $O(C_{p_n/2}n^{\frac{p_n}{2}- k\epsilon})$
\end{prop}

\begin{proof}
First notice that the number of loops must be even as the odds moments of a Gaussian random variable are zero.  As in proposition (\ref{prop:ver}), we let $T_i$ denote the number of times the path stays on level $i$.  The number of paths with $2k$ loops are bounded by
\begin{equation}\label{eqn:loop}
\frac{1}{\frac{p_n}{2}-k+1}\binom{p_n-2k}{\frac{p_n}{2}-k} \sum_{2a_1,\hdots,2a_l: \sum_1^l 2a_i = 2k} \binom{T_i}{2a_i},
\end{equation}
since the loops have to be paired on the 'same level' and removing the loops give a Dyck path of length $p_n-2k$.  With overwhelming
probability, each path contributes $O(n^{\frac{p_n}{2}-k}\prod_i(2a_i-1)!!)$.  Indeed, if we remove the $2k$ loops, it reduces to a Dyck path that the earlier analysis applies, hence the factor $n^{\frac{p_n}{2}-k}$ and each loop on the same level contributes the $2a_i$-th moment of the Gaussian random variable, which is of the order $(2a_i-1)!!$.  For a given $k$, the contribution from the terms in equation (\ref{eqn:loop}) is dominated by the term where $l=k$ and $a_1=\hdots=a_l=2$ (one way to see this is that this is a similar argument as the one which shows that trace of moments of Wigner matrices are dominated by paths that traverse each edge twice after considering only the contribution from the loops, see e.g. \cite{SS}, the key here being the moments of Gaussian does not grow too fast; another way is to calculate the contribution when $l < k$.  This is bounded by $\binom{n^{1/3-\epsilon}}{l} l^{k-l} n^{2k/3-k\epsilon} (\lceil\frac{k}{l}\rceil!!)^l$, the binomial term coming from picking $l$ levels for the loops, the second term being each level must be filled by at least two loops, and the remaining loops can be anywhere, the third term is just the maximum possible of $\prod_i\binom{T_i}{a_i} (2a_i-1)!!$ and as explained below, $\max_iT_i < n^{1/3-\epsilon}$.  Summing over all $l<k$ yields the desired dominance.) and is bounded by $n^{k/3-k\epsilon}n^{2k/3-k\epsilon}$ for $p_n = O(n^{2/3-\epsilon})$, since we have $\max_i T_i = O(n^{1/3-\epsilon})$ with overwhelming probability and therefore
$\binom{T_i}{a_i} = O(n^{a_i/3-\epsilon})$ with overwhelming probability and moreover, by
proposition (\ref{prop:hor}), the number of different $T_i$ to choose from in the sum is of $O(n^{1/3-\epsilon})$ with overwhelming probability and so the number of choice is bounded by $n^{k/3-k\epsilon}$.  The exceptional paths can be dealt with similar to the ones without loops.  Using the trivial bound of $C_{p_n/2-k}<C_{p_n}$ result in the proposition.  
\end{proof} 

If we sum over all $k\geq 2$ in the proposition above, we found that the contribution from the paths with loops is of the order $O(C_{p_n/2}n^{p_n/2-\epsilon})$.

Combining all of the above, we have that
$\mathbb{E}(A^{p_n}_{n\beta})_{11}$ is given by 
\begin{displaymath}
 C_{p_n/2} n^{p_n/2} (1+o(1))
\end{displaymath}
when $p_n$ is even and $0$ when $p_n$ is odd.  Here $C_k$ denotes the k-th
Catalan number, which counts the number of Dyck paths between $0$ and $2k$.  It
is well known that:
\begin{displaymath}
 C_k \sim \frac{4^k}{k^{3/2}\sqrt{\pi}}
\end{displaymath}
 and the result follows from lemma \ref{lem:simp} after we rescaled by
$2^{p_n}n^{p_n/2}$.

\section{Conclusion}
In this article we prove a local semicircle law at the edge of the spectrum for
$\beta$-ensembles.  The proof relies on calculations of moments and can be
applied to other tridiagonal matrices satisfying fairly mild conditions: that
the diagonal entries remain $O(1)$ and the growth of the moments of the off
diagonal entries is not too fast.  It would be interesting to consider the same
question for the $\beta$-Laguerre ensembles.

\section{Acknowledgements}
 The author would like to thank his advisor Prof. Ya. G. Sinai for his
suggestion of the problem and his continual support and Prof. Alice Guillonet for providing useful remarks and discussions.  Finally, the author would like to thank the referee for pointing out various typos and making a number of suggestions to improve the article.

\section{Appendix}
We shall provide the derivation of
theorem \ref{thm:main} from theorem \ref{thm:mom} for the readers' convenience.

\begin{proof}{Proof of theorem \ref{thm:main}}
 To prove the theorem, it suffices to establish the convergence of the Laplace
transform of the measures $\mu_n$ to that of $\mu$ (see e.g. \cite{F2} for the relationship between convergence of Laplace transform and convergence in distribution):
\begin{displaymath}
 \int_{-\infty}^\infty e^{-c\theta}d\mu_n(\theta) =
\frac{1}{nr_n^{3/2}}\sum_{k=1}^n e^{-c\theta_k}
\rightarrow_{n\rightarrow\infty} \int_0^\infty
e^{-c\theta}\frac{2\sqrt{2}}{\pi} \sqrt{\theta} d\theta = \sqrt{\frac{2}{\pi
c^3}}
\end{displaymath}
Let $s_n = cr_n^{-1}/2$ and $p_n = 2s_n$.  We claim that
the sequence $n^{-1} r_n^{3/2} \textrm{Tr} A^{2s_n}$ converges in
probability to 
\begin{displaymath}
 \lim_{n\rightarrow\infty} \frac{1}{nr_n^{3/2}} \mathbb{E}(\textrm{Tr}
A^{2s_n}) =
\lim_{n\rightarrow\infty} \frac{1}{nr_n^{3/2}} \frac{n}{\sqrt{\pi s_n^{3}}} =
\frac{2\sqrt{2}}{\sqrt{\pi c^3}}
\end{displaymath}
This follows from Chebyshev's inequality once we show that the variance of
$n^{-1} r_n^{3/2} \textrm{Tr} A^{2s_n}$ is of $o(1)$.  This is true as:
\begin{align} \nonumber
 \textrm{var}(n^{-1} r_n^{3/2} \textrm{Tr} A^{2s_n})& =
n^{-2}r_n^{-3} \textrm{var}(\textrm{Tr} A^{2s_n})\\ \nonumber
& \leq n^{-1}r_n^{-3} \mathbb{E}((\textrm{Tr} A^{2s_n})^2) \leq
2n^{-1}r_n^{-3}\mathbb{E}(\textrm{Tr}
A^{4s_n})
\end{align}
By theorem \ref{thm:mom}, the last quantity is of $o(1)$.  Similarly,
$n^{-1}r_n^{-3/2} \textrm{Tr} A^{2s_n+1}$ converges in probability to
zero.
Let 
\begin{align}
 \lambda_k = 1 - \theta_k r_n & \lambda_k \geq 0 \\
 \lambda_j = -1 + \tau_j r_n & \lambda_j \leq 0
\end{align}
Then by definition, we have
\begin{equation}\label{eqn:tr1}
 \textrm{Tr} A^{2s_n} = \sum_k (1-r_n \theta_k)^{2((c/2)r_n^{-1})} +
\sum_j
(1-\tau_jr_n)^{2((c/2)r_n^{-1})}
\end{equation}
and
\begin{equation}\label{eqn:tr2}
 \textrm{Tr} A^{2s_n+1} = \sum_k (1-r_n \theta_k)^{2((c/2)r_n^{-1})+1} -
\sum_j
(1-\tau_jr_n)^{2((c/2)r_n^{-1})+1}
\end{equation}
If $|\theta_k| \leq r_n^{-1/3}$ and $|\tau_j| \leq r_n^{-1/3}$, the
corresponding terms in (\ref{eqn:tr1}) and (\ref{eqn:tr2}) are
$e^{-c\theta_k}(1+o(r_n^{1/3}))$ and $e^{-c\tau_j}(1+o(r_n^{1/3}))$.  On the
other hand, using the estimate of the expectation of $\textrm{Tr}
A^{4s_n}$, the subsums
in (\ref{eqn:tr1}) and (\ref{eqn:tr2}) over $\theta_k$ and $\tau_j$ such that
$|\theta_k| > r_n^{1/3}$ and $|\tau_j| > r_n^{1/3}$ converge to zero in
probability.  The same holds for the subsums over these $\theta_k$ and $\tau_j$
of $\sum_k e^{-c\theta_k}$ and $\sum_j e^{-c\tau_j}$.  Therefore, we have
\begin{displaymath}
 (\textrm{Tr} A^{2_sn} + \textrm{Tr} A^{2s_n+1} - 2\sum_k
e^{-c\theta_k})\frac{1}{nr_n^{3/2}}
\rightarrow 0
\end{displaymath}
and
\begin{displaymath}
 \frac{1}{nr_n^{3/2}}\sum_k e^{-c\theta_k} \rightarrow \sqrt{\frac{2}{\pi c^3}}
\end{displaymath}
If $r_n = n^{-\gamma}$ where $\gamma < 2/3$, then it follows from theorem \ref{thm:mom} that the normalized traces $n^{-1}r_n^{-3/2}\textrm{Tr}A^{2s_n}$ and $n^{-1}r_n^{-3/2}\textrm{Tr}A^{2s_n+1}$ converges to nonrandom limits with probability $1$, and hence the convergence of the measures $\mu_n$ in the main theorem also occurs with probablity $1$. 
\end{proof}


\begin{thebibliography}{100}
\bibitem{AGZ} Anderson, G.W., Guionnet, A. and Zeitouni, O.: An Introduction to
Random Matrices, \emph{Cambridge Univ. Press} (2010)
\bibitem{BS} Bao, Z., and Su, Z.: Local semicircle law and Gaussian fluctuation for Hermite $\beta$-ensemble. (2011) Preprint, arXiv:math.PR/1104.3431
\bibitem{BEY} Bourgarde, P., Erd\"os, L., and Yau, H.-T.: Universality of General $\beta$-ensemble. (2011) Preprint, arXiv:math.PR/1104.2272
\bibitem{D} Dumitriu, I.: Eigenvalue statistics for beta ensembles, \emph{Ph.D.
thesis, MIT} (2003)
\bibitem{DE} Dumitriu, I. and Edelman, A.: Matrix models for beta ensembles,
\emph{J Math. Phys.} \textbf{43} (11), 5830-5847 (2002)
\bibitem{ES} Edelman, A. and Sutton, B.: From random matrices to stochastic
operators \emph{J. Stat. Phys.} \textbf{127} (6), 1121-1165 (2007)
\bibitem{ESY} Erd\"os, L., Schlein, B. and Yau, H.-T.: Semicircle law on short
scales and delocalization of eigenvectors for Wigner random matrices \emph{Ann.
Probab.} \textbf{37}(3), 815-852 (2008)
\bibitem{ESY2} Erd\"os, L., Schlein, B., and Yau, H.-T.: Local semicircle law
and complete delocalization for Wigner random matrices. \emph{Commun. Math.
Phys.} \textbf{287}, 641-655 (2009)
\bibitem{F2} Feller, W., An introduction to Probability Theory and its Applications, Vol. 2 \emph{Wiley and Sons} (1971)
\bibitem{F} Forrester, P.: Log-gases and Random Matrices, \emph{Princeton Univ.
Press} (2010) 
\bibitem{KM} Khorunzhiy, O. and Marckert, J.-F.: Uniform bounds for exponential
moment of maximum of a dyck path, \emph{Elect. Comm. in Probab.} \textbf{14}
 327-333 (2009)
\bibitem{R} Revesz, P.: Random walk in random and non-random environments,
\emph{World Scientific} (2005)
\bibitem{RRV} Ram\'irez, J., Rider, B. and Virag, B: Beta ensembles, stochastic
Airy spectrum and a diffusion.  (2007) Preprint, arXiv:math.PR/0607331.
\bibitem{SS} Sinai, Ya. G. and Soshnikov, A.B.: A refinement of Wigner's
semicircle law in a neighborhood of the spectrum edge for random symmetric
matrices, \emph{Funct. Anal. and Its Appl.} \textbf{32}(2) (1998)
\bibitem{S} Soshnikov, A.B.: Universality at the edge of the spectrum in Wigner
random matrices, \emph{Comm. Math. Phys.} \textbf{207} 697-733 (1999) 
\bibitem{TV} Tao, T. and Vu, V.: Random matrices: Universality of local
eigenvalue statistics, (2009) Preprint, arXiv:0906.0510v10
\bibitem{T} Trotter, H.: Eigenvalue distributions of large Hermitian matrices;
Wigner's semicircle law and a theorem of Kac, Murdoch and Szego. \emph{Adv. in
Math.} \textbf{54}(1) 67-82 (1984) 
\bibitem{TW} Tracy, C., and Widom, H.: Level spacing distributions and the Airy
kernel. \emph{Comm. Math. Phys.} \textbf{159}(1) 151-174 (1994) 
\bibitem{TW2} Tracy, C., and Widom, H.: On orthogonal and symplectic matrix
ensembles. \emph{Comm. Math. Phys.} \textbf{177}(3) 727-754 (1996)  
\end{thebibliography}
\end{document}